\documentclass[twoside,a4paper,12pt]{amsart}
\usepackage{amsfonts}
\usepackage{amssymb}
\usepackage{amsfonts}
\usepackage{amssymb}
\usepackage{amsthm}
\usepackage{amsmath}
\usepackage{a4wide}
\usepackage{amsthm}
\usepackage[usenames,dvipsnames,svgnames,table]{xcolor}

\theoremstyle{plain}
\newtheorem{teo}{Theorem}[section]

\newtheorem{Corollary}[teo]{Corollary}
\newtheorem{ackn}{Acknowledgement\!}

\theoremstyle{definition}

\theoremstyle{remark}
\newtheorem{Remark}[teo]{Remark}

\newcommand{\intF}{\overset{\circ}{F}}
\newcommand{\mbd}{{\partial}}

\newcommand{\LL}{\mathcal{L}}

\def\R{\mathbb R}

\renewcommand{\d}{\delta }

\newcommand{\e }{\varepsilon }
\newcommand{\g }{\gamma}

\newcommand{\MM}{\mathcal{M}}

\newcommand{\s }{\sigma }

\renewcommand{\th }{\theta }

\newcommand{\intbar}{\etaathop{\int\etaakebox(-13.5,0){\rule[4pt]{.7em}{0.3pt}}%
\kern-6pt}\nolimits}

\newcommand{\be}{\begin{equation}}
\newcommand{\ee}{\end{equation}}
\newcommand{\bea}{\begin{equation*}}
\newcommand{\eea}{\end{equation*}}

\newcommand{\grad}{\nabla}

\newcommand{\mmm}{\boldsymbol{\textrm{m}}}
\newcommand{\PPP}{\textrm{P}}

\theoremstyle{plain}

\newtheorem{lemma}[teo]{Lemma}
\newtheorem{prop}[teo]{Proposition}

\theoremstyle{definition}
\newtheorem{dfnz}[teo]{Definition}

\theoremstyle{remark}

\def\R{{{\mathbb R}}}

\def\H{{{\mathcal H}}}

\def\C{\mathcal{C}}
\def\HHH{{\mathrm H}}

\def\F{{\mathcal F}}

\renewcommand{\d}{\delta }

\def\be{\begin{equation}}
\def\ee{\end{equation}}
\def\bea{\begin{eqnarray*}}
\def\bean{\begin{eqnarray}}
\def\eean{\end{eqnarray}}
\def\eea{\end{eqnarray*}}

\author[A. Magni]
{Annibale Magni}
\address[Annibale Magni]{Westf\"alische--Wilhelms Universit\"at M\"unster, Einsteinstrasse 62, 48149 M\"unster (Germany).}
\email[]{magni@uni-muenster.de}
\author[M. Novaga]
{Matteo Novaga}
\address[Matteo Novaga]{Dipartimento di Matematica
Universit\'a di Pisa, Largo Bruno Pontecorvo 5, 56127 Pisa (Italy).}
\email[]{novaga@dm.unipi.it}
\title
[Non Lower Semicontinuous Perimeter Functionals]
{A note on non lower semicontinuous perimeter functionals on partitions}
\begin{document}
\begin{abstract}
We consider isotropic non lower semicontinuous weighted perimeter functionals defined on partitions of domains in $\R^n$. Besides identifying a condition on the structure of the domain which ensures the existence of minimizing configurations, we describe the structure of such minima, as well as their regularity.
\end{abstract}
\maketitle

\section{Introduction}
In this note we address the existence, the structure and the regularity properties of minimizing configurations for weighted non lower semicontinuous perimeter functionals of the form
\begin{equation} \label{eq:Functional}
\F_{\Omega,\mmm}(E_1,E_2,E_3):= \sum_{i < j \in \{1,2,3\}} \s_{ij} \H^{n-1}(\partial^*E_i \cap \partial^*E_j \cap \Omega)\,,
\end{equation}
defined on partitions of a domain $\Omega \subset \R^n$ in three sets with prescribed Lebesgue measures $\mmm := (m_1,m_2,m_3)$, where $\s_{ij} > 0$. This functionals arise in the modelling of multicomponent systems interacting at the contact interfaces via isotropic energies. Their diffuse approximation as well as methods to study their diffuse gradient dynamic have been recently considered in \cite{EsOt14} and \cite{BeMa15}. We believe that non lower semicontinuous functionals as \eqref{eq:Functional} can represent a good model to describe from a macroscopic point of view the effect that surface tension has in selecting equilibrium configurations of biological cell sorting phenomena for two species. A rigorous microscopic cellular description of these phenomena have been given amongst others in \cite{VoDe10} (see also references therein), whose results are in accordance with the ones in our work. Identifying the sets $E_1$ and $E_2$ with the regions in the domain $\Omega$ which are occupied by the two cell species and denoting by $E_3$ the remaining environment in which the cells can move, we can find a stable condition (equation \eqref{eq:violation_triangle}) on the three surface tensions $\s_{ij}$ under which the minima of the functional exhibit separation between one of the cell types and the environment. For a particular class of domains, including the most significant biological cases (see Definition \ref{dfnz:isop-Foliation}), we can also describe quite explicitly the shape which is taken by each region in correspondence of a minimum (Proposition \ref{prop:existence-minimizers}).\\
The first results on weighted perimeter functionals on partitions have been proven with different methods in \cite{AmBr90} and \cite{Whit96}, where it is shown that (independently of the number of the sets in the partition) \eqref{eq:Functional} is lower semicontinuous if and only if
\begin{equation} \label{eq:triangle}
\s_{ij} \leq \s_{ki} + \s_{kj} \quad \textrm{for all}\,\, i \neq j \neq k \in \{1,2,3\}\,. 
\end{equation}
As for the regularity of the minimizers, the strict triangle inequalities in \eqref{eq:triangle} are sufficient to prove that $\H^{n-1}-$almost every point of the minimizing interfaces belongs to the boundary of just two elements of the partition (see \cite{Whit96} and \cite{Leon01}) and this allows to apply the standard regularity results for minimizing boundaries with prescribed volume.\\
Some topological properties of the minimizers of \eqref{eq:Functional} in the case of non lowersemicontinuity (i.e. when \eqref{eq:triangle} is violated) have been described in \cite{Morg98} and, under the hypothesis of existence, some results on the regularity of the minimizing boundaries have been announced in \cite{Whit96}.\\
In the first part of our note we introduce the class of domains $\Omega \subset \R^n$ which are foliated by isoperimetric sets (see Definition \ref{dfnz:isop-Foliation} below) and we prove that on these domains the functional \eqref{eq:Functional} has always a minimum. In the second part we show that if \eqref{eq:Functional} admits a minimizer, then, independently of the domain $\Omega$, two of the phases do not come in contact and this allows us to prove regularity for the boundaries of the minimizing configurations. We close the paper describing a situation in which \eqref{eq:Functional} has no minimizer.
\begin{ackn} \rm
The authors thank A. Stevens (M\"unster) for inspiring this research by pointing them out (cell)
sorting phenomena in bio-chemistry and biology, as well as their mathematical modelling. The authors also thank G.P. Leonardi for some useful comments on a preliminary version of this paper.\\
The research of A. Magni was sponsored by the Excellence Cluster "Cells in Motion" (CiM), M\"unster: FF-2013-29. The work of M. Novaga was partly supported by the Italian CNR-GNAMPA and by the University of Pisa via grant PRA-2015-0017
\end{ackn}

\section{Notation}
With $\LL^n$ we denote the $n$-dimensional Lebesgue measure. Given a set $F \subset \R^n$ with finite perimeter, we denote its perimeter with $\PPP(F)$ and its relative perimeter with respect to an open set $\Omega \subset \R^n$ with $\PPP(F;\Omega)$. By $\chi_F$ we denote the characteristic function of $F$, by $\partial^* F$ its reduced boundary and by $|F|:= \int_{\R^n}\chi_F(x) dx$ its volume. For $\LL^n$-almost all points $x \in \R^n$ the density at $x$ with respect to the Lebesgue measure of a set $F \subset \R^n$ having finite perimeter is denoted by
\begin{equation*}
\th_E(x) := \lim_{r \downarrow 0} \frac{|B_r(x) \cap E|}{|B_r(x)|}\,,
\end{equation*}
where $B_r(x) \subset \R^n$ is the Euclidean ball with center $x$ and radius $r$.\\
In order to have notions of boundary, closure and interior for a set $F \subset \R^n$ with finite perimeter, which are invariant under $\LL^n$ negligible changes, we define the measure theoretic boundary, closure and interior part respectively as
\begin{equation*}
\mbd F := \{ x \in \R^n: \forall r> 0 \quad |F \cap B_r(x)| \notin \{0,|B_r(x)|\}\}\,,
\end{equation*}
\begin{equation*}
\overline{F} := \{ x \in \R^n: \forall r> 0 \quad |F \cap B_r(x)| \neq 0 \}\,,
\end{equation*}
\begin{equation*}
\intF := \{ x \in \R^n: \exists r> 0 \quad |F \cap B_r(x)| = |B_r(x) \}\,.
\end{equation*}
If $\partial^*F$ is sufficiently regular, $\HHH_{\partial^*F}(x)$ denotes the scalar mean curvature of $\partial^*F$ at $x \in \partial^*F$ (i.e. the sum of the principal curvatures of the surface at the point $x$).\\
Let $\Omega \subset \R^n$ be an open set and let $\mmm:= (m_1,m_2)$, with $0<m_1,m_2<\infty$. Let also $E_1, E_2, E_3 \subset \Omega$ be three sets with finite perimeter. We say that $(E_1,E_2,E_3)$ belongs to $\C_{\Omega,\mmm}$ (the set of admissible test configurations) if $|E_i \cap E_j| = 0$ for all $i<j \in \{1,2,3\}$, $|E_1|=m_1$, $|E_2|=m_2$ and $|\Omega \setminus (E_1 \cup E_2 \cup E_3)| = 0$.\\
For $i < j \in \{1,2,3\}$ consider $\s_{ij} > 0$ such that
\be \label{eq:violation_triangle}
\s_{13} > \s_{23} + \s_{12} \,.
\ee
On the set $\C_{\Omega,\mmm}$ we define the functional
\be \label{eq:main_functional}
\F_{\Omega,\mmm}(E_1,E_2,E_3) := \sum_{i < j \in \{1,2,3\}} \s_{ij} \H^{n-1}(\partial^*E_i \cap \partial^*E_j \cap \Omega)\,.
\ee
The set of triples $(F_1,F_2,F_3) \in \C_{\Omega,\mmm}$ minimizing $\F_{\Omega,\mmm}$ will be denoted by $\MM(\F_{\Omega,\mmm})$.\\
%
%
%

\noindent We are interested at existence and regularity of minimizers of 
$\F_{\Omega,\mmm}$ on $\C_{\Omega,\mmm}$.

\section{Existence of minimizers in foliated domains}
We now define a class of domains on which the functional \eqref{eq:main_functional} admits a minimum for any choice of $\mmm$.
\begin{dfnz} \label{dfnz:isop-Foliation}
Let $\Omega \in \R^n$ be an open set with (possibly infinite) Lebesgue measure $|\Omega|$ and suppose that for any $0 < \rho < |\Omega|$ there exists a minimizer for the problem
\be \label{eq:isop-Omega}
\min\,\Big\{\PPP(E;\Omega):\ E \subset \Omega,\, |E|=\rho \Big\}\,.
\ee
We say that $\Omega$ is an {\it isoperimetrically foliated domain} if there exists a selection of minimizers $E^{(\rho)}$ of \eqref{eq:isop-Omega} such that
\begin{equation} \label{eq:isop-Foliation}
\rho_1 < \rho_2 \in (0,|\Omega|) \Longrightarrow \begin{cases} E^{(\rho_1)} \subset E^{(\rho_2)} \\
                                                               \H^{n-1}(\partial E^{(\rho_1)}\cap \partial E^{(\rho_2)} \cap \Omega) = 0
                                                 \end{cases}
\end{equation}
\end{dfnz}
\begin{Remark}\label{remo}\rm
By standard regularity theory (see \cite{Magg12}), $\partial E^{(\rho)}$ is $C^{\infty}$ away from a closed singular set of Hausdorff dimension at most $n-8$.\\
We also observe that if $\Omega$ is an isoperimetrically foliated set as in Definition \ref{dfnz:isop-Foliation}, then solutions to problem \eqref{eq:isop-Omega} which satisfy \eqref{eq:isop-Foliation} foliate $\Omega$ with their boundaries, i.e.
\be \label{eq:isop-Foliation-bd}
\Omega = \bigcup_{0 < \rho < |\Omega|} (\partial E^{(\rho)} \cap \Omega) \,.
\ee
\end{Remark}
\begin{Remark} \label{rm:examples_domains}\rm
Examples of isoperimetrically foliated domains are $\R^n$ itself, ellipses in $\R^2$ and $B^n$.\\
Notice that the unit square in $\R^2$ is not an isoperimetrically foliated domain, as for $\rho \in [0,\frac{1}{\pi}]$ the corresponding $E^{(\rho)}$ is a quarter of disk centred at any of the four vertices, while, for $\rho \in [\frac{1}{\pi}, 1 - \frac{1}{\pi}]$, $E^{(\rho)}$ is a vertical (or horizontal) stripe.
\end{Remark}
Under some assumption on the regularity of $\partial \Omega$, the existence of a selection of isoperimetric sets $E^{(\rho)}$ with the property
$$\rho_1 < \rho_2 \in (0,|\Omega|) \Longrightarrow E^{(\rho_1)} \subset E^{(\rho_2)}
$$
is sufficient to conclude that $\H^{n-1}(\partial E^{(\rho_1)}\cap \partial E^{(\rho_2)} \cap \Omega) = 0$.
\begin{prop} \label{prop:boundary-separation}
Within the setting of Definition \ref{dfnz:isop-Foliation}, assume that there exists a $\rho_0 > 0$, such that for all $\rho < \rho_0$ the sets $E^{(\rho)}$ are connected, with connected boundary. If
$$\rho_1 < \rho_2 \in (0,|\Omega|) \Longrightarrow E^{(\rho_1)} \subset E^{(\rho_2)}\,,$$
then $\H^{n-1}(\partial E^{(\rho_1)}\cap \partial E^{(\rho_2)} \cap \Omega) = 0$.
\end{prop}
\begin{proof}
The assumptions on $\Omega$ ensure that the relative isoperimetric inequality holds, i.e. there exists a $C_{\Omega} > 0$ such that $P(E,\Omega)\geq C_\Omega \min(|E|^\frac{n-1}{n}, |\Omega\setminus E|^\frac{n-1}{n})$. Moreover, all the sets $E^{(\rho)}$ have $C^{\infty}$ boundary on the complement of a closed set of Hausdorff dimension at most $n-8$ and on the regular part of the boundary there holds $H_{\partial E^{(\rho)}} = \lambda_{\rho}$, for a $\lambda_{\rho}\in\R$.\\
We claim that $E^{\rho}$ and $\partial E^{(\rho)}$ are connected for all $\rho \in [0,|\Omega|)$. By hypothesis, this is true for all $\rho < \rho_0$. Let $\rho_M$ be the supremum of the set of values for which this property holds and assume that $\rho_M < |\Omega|$. The existence of a minimizer for \eqref{eq:isop-Omega} for sets with mass equal to $\rho_M$, together with the relative isoperimetric inequality, would imply that $\partial_{\rho} P(E^{(\rho)};\Omega)|_{\rho=\rho_M} = + \infty$, and this would contradict that $\partial_{\rho} P(E^{(\rho)}; \Omega) |_{\rho = \rho_M} = \lambda_{\rho} \in \R$ on the regular part of $\partial E^{(\rho)}$ and the claim is proven.\\
Suppose now by contradiction that there exist $\rho_1 < \rho_2$ such that $\mathcal H^{n-1}(\partial E^{(\rho_1)}\cap\partial E^{(\rho_2)}\cap\Omega)>0$. By the strong maximum principle it would follow that $\lambda_{\rho_1}=\lambda_{\rho_2}$ and that $\partial E^{(\rho_1)}\cap\Omega$ and $\partial E^{(\rho_2)}\cap\Omega$ have a common non empty connected component with non zero $n-1$ Hausdorff measure. In view of the previous claim, this would imply that $E^{(\rho_1)} = E^{(\rho_2)}$, but this is impossible.
\end{proof}
\begin{Remark} \rm
The assumption in the statement of Proposition \ref{prop:boundary-separation} are clearly satisfied if $\Omega$ is bounded with $C^1$ boundary and  in the case of a convex $\Omega$ (see \cite{RiVe15}, Proposition 6.6). This follows from tha fact that in these cases, for $\rho > 0$ amall enough, $E^{(\rho)}$ is close to the intersection of $\Omega$ with a ball centered at a point of $\partial\Omega$.
\end{Remark}
\begin{teo} \label{prop:existence-minimizers}
Let $\Omega \subset \R^n$ be an isoperimetrically foliated domain and let $\mmm:= (m_1,m_2)$, with $0<m_1,m_2<\infty$ and $m_1 + m_2 < |\Omega|$. 
Then $\F_{\Omega,\mmm}$ attains its minimum in $\C_{\Omega,\mmm}$.
\end{teo}
\begin{proof}
Let $(E_1,E_2,E_3) \in \C_{\Omega,\mmm}$. By means of \eqref{eq:violation_triangle} and \eqref{eq:main_functional}, we obtain
\be
\begin{split}
\F_{\Omega,\mmm}(E_1,E_2,E_3) & \geq (\s_{12} + \s_{23}) \H^{n-1}(\partial^*E_1 \cap \partial^*E_3 \cap \Omega) + \s_{12} \H^{n-1}(\partial^*E_1 \cap \partial^*E_2 \cap \Omega)\\
                                & \quad + \s_{23} \H^{n-1}(\partial^*E_2 \cap \partial^*E_3 \cap \Omega) \\
                                & = \s_{12} \PPP(E_1;\Omega) + \s_{23} \PPP(E_3;\Omega)\\
                                & \geq \s_{12} \PPP(E^{(m_1)};\Omega) + \s_{23} \PPP(E^{(m_1 + m_2)};\Omega)\,.
\end{split}
\ee
Since $\Omega$ is isoperimetrically foliated, the infimum of \eqref{eq:main_functional} is attained for the admissible choice $F_1 := E^{(m_1)}$, $F_2 := E^{(m_1 + m_2)} \setminus E^{(m_1)}$ and $F_3 := \Omega \setminus E^{(m_1 + m_2)}$.
\end{proof}

\begin{Remark}\rm
Thanks to the regularity of $\partial E^{(\rho)}$ (see Remark \ref{remo}),
the minimizing sets $F_1,\,F_2,\,F_3$ constructed in Theorem 
\ref{prop:existence-minimizers} have boundaries of class
$C^{\infty}$ away from a closed singular set of Hausdorff dimension at most $n-8$.
\end{Remark}

\begin{Remark}\rm
It is easy to see that the existence of an isoperimetric foliation of $\Omega$ is sufficient but not necessary for the existence of minimizers of \eqref{eq:main_functional}. In view of Remark \ref{rm:examples_domains}, if $\Omega = (0,1) \times (0,1) \subset \R^2$, $m_1 < \pi/16$, $m_3 < \pi/16$ and $m_2 = 1 - m_1 - m_3$, the minimum of \eqref{eq:main_functional} is attained for $F_1 = \{x^2 + y^2 < 4 m_1 / \pi\} \cap \Omega$, $F_3 = \{(x-1)^2 + (y-1)^2 < 4 m_2 / \pi \} \cap \Omega$ and $F_2 = \Omega \setminus \overline{(F_1 \cup F_2)}$.
\end{Remark}
\section{Regularity of minimizers in general domains}

We state a result which has been stated in \cite{Whit96} in a slightly different form and whose proof is an easy modification of the one of Theorem 3.1 in \cite{Leon01}.
\begin{teo} \label{thm:no-infiltration}
Suppose that $(F_1,F_2,F_3) \in \MM(\F_{\Omega,\mmm})$. Then there exist $\eta , r > 0$ such that, for all $\rho < r$, $x \in \R^n$ and $k \in \{1,3\}$ holds
\be \label{eq:no-infiltration}
|F_k \cap B_{\rho}(x)| < \eta \rho^n \quad \Rightarrow \quad |F_k \cap B_{\rho / 2}(x)| = 0 \,.
\ee
\end{teo}
We now prove a result on the structure of the minimizers of \eqref{eq:main_functional}, which does not depend on the fact that the domain $\Omega$ is isoperimetrically foliated.
\begin{lemma} \label{lem:separation123}
Let $\Omega \subset \R^n$ be an open set and let $\mmm:= (m_1,m_2)$, with $0<m_1,m_2<\infty$ and $m_1 + m_2 < |\Omega|$. If $(F_1,F_2,F_3) \in \MM(\F_{\Omega,\mmm})$, then for every $x\in \partial^*F_2\cap \partial^*F_1 \cap \Omega$ (resp. $x\in \partial^*F_2\cap \partial^*F_3 \cap \Omega$) there exists $r>0$ such that 
\be\label{eq:separation123}
B_r(x)\cap F_3 = \emptyset \qquad \left(resp.\ B_r(x) \cap F_1 = \emptyset \right).
\ee
\end{lemma}
\begin{proof}
We consider $x \in \partial^*F_2 \cap \partial^*F_1 \cap \Omega$, 
since the other case follows by the same argument. By the definition of reduced boundary it follows that $x$ has density $1/2$ with respect to both $F_1$ and $F_2$. Thus, for a sufficiently small $r_0 > 0$ it holds $|F_3 \cap B_{r_0}(x)| < \eta r^n_0$ and Theorem \ref{thm:no-infiltration} ensures that $|F_3 \cap B_{r_0/2}(x)| = 0$. Consequently \eqref{eq:separation123} holds with $0 < r \leq r_0/2$.
\end{proof}
\begin{prop}
Let $\Omega \subset \R^n$ be an open set and let $\mmm:= (m_1,m_2)$, with $0<m_1,m_2<\infty$ and $m_1 + m_2 < |\Omega|$. If $(F_1,F_2,F_3) \in \MM(\F_{\Omega,\mmm})$, then there exists a constant $\g \in (0,1)$ such that, for every $\mbd F_1$ (resp. $x \in \mbd F_3$), it holds
\begin{equation} \label{eq:density_boundary}
\g \leq \frac{|F_1 \cap B_r(x)|}{\omega_n r^n} \leq 1-\g
\end{equation}
for all $r > 0$ such that $B_r(x) \subset \subset \Omega$.
\end{prop}
\begin{proof}
The inequality on the left-hand side of \eqref{eq:density_boundary} follows immediately from \eqref{eq:no-infiltration} applied to $\mbd F_1$ (resp. $\mbd F_3$) and it follows that $\eta \leq \g$.\\
We claim that the inequality on the right-hand side of \eqref{eq:density_boundary} holds for a $\g \geq \eta$ and we consider the case of $x \in \mbd F_1$, being the other case identical. Suppose that there exists $x \in \mbd F_1$, such that $\frac{|F_1 \cap B_r(x)|}{\omega_n r^n} > 1-\eta$. This implies that $\frac{|F_3 \cap B_r(x)|}{\omega_n r^n} < \eta$ and consequently (by \eqref{eq:no-infiltration}) that $|F_3 \cap B_{r/2}(x)| = 0$. Thus we conclude that $x \in \mbd F_1 \cap \mbd F_2$ and the standard regularity results for minimizing boundaries with fixed volume apply to give the desired estimate.
\end{proof}
Thanks to \eqref{eq:density_boundary}, arguing as in Proposition 3.5 of \cite{ANPa02} (see also \cite{ATWa93}, 3.4), we obtain the following estimate from below for the $n-1$ density of the minimizing boundaries. 
\begin{prop}
There exist $\theta > 0$ and $\overline{r} > 0$ such that, if $(F_1,F_2,F_3) \in \MM(\F_{\Omega,\mmm})$, for all $x \in \mbd F_1$ (resp. $x \in \partial F_3$) and for all $B_r(x) \subset \subset \Omega$, with $r < \overline{r}$ it holds
\begin{equation} \label{eq:per-geq-hausd}
P(F_1,B_r(x)) \geq \theta r^{n-1}\,.
\end{equation}
\end{prop}
\begin{Corollary}
If $(F_1,F_2,F_3) \in \MM(\F_{\Omega,\mmm})$, it holds
\begin{equation} \label{eq:minimal-regularity}
\H^{n-1}((\mbd F_1 \cap \Omega) \setminus (\partial^* F_1 \cap \Omega)) = 0 \quad \quad \textrm{and} \quad \quad \H^{n-1}((\mbd F_3 \cap \Omega) \setminus (\partial^* F_3 \cap \Omega)) = 0 \,.
\end{equation}
\end{Corollary}
\begin{proof}
We prove the claim just for $F_1$, since the proof for $F_3$ is identical. For any Borel set of $\R^n$, with $B \subset \mbd F_1$, by \eqref{eq:per-geq-hausd}, we have
\begin{equation*}
|\grad \chi_{F_1}|(B) \geq \frac{\theta}{\omega_{n-1}} \H^{n-1}(B)\,,
\end{equation*}
where $|\grad \chi_{F_1}|$ is the total variation measure associated to $\chi_{F_1}$. With the choice $B := \mbd F_1 \setminus (\partial^* F_1 \cap \Omega)$, since $|\grad \chi_{F_1}|$ is concentrated on $\partial^* F_1$, the thesis follows.
\end{proof}
\begin{prop}  \label{prop:separation-phases}
Let $\Omega \subset \R^n$ be an open set and let $\mmm:= (m_1,m_2)$, with $0<m_1,m_2<\infty$ and $m_1 + m_2 < |\Omega|$. 
For any $(F_1,F_2,F_3) \in \MM(\F_{\Omega,\mmm})$ it holds
\be \label{eq:separation-phases}
\H^{n-1}(\partial^*F_1 \cap \partial^*F_3 \cap \Omega) = 0\,.
\ee
\end{prop}
\begin{proof}
Since \eqref{eq:minimal-regularity} holds, we can identify 
$F_1$ with $\intF_1$ and $F_3$ with $\intF_3$.
Suppose by contradiction that $\H^{n-1}(\partial^*F_1 \cap \partial^*F_3) > 0$ and, for an $\e > 0$, we approximate $F_1$ from the inside with a smooth set $F^{\e}_1$ as in \cite{Schm15}, in such a way that
\begin{equation} \label{eq:perim-approx}
|P(F_1,\Omega) - P(F^{\e}_1,\Omega)| < \e \,.
\end{equation}
We now define $F^{\e}_3 := F_3$, $F^{\e}_2 := \Omega \setminus (\overline{F^{\e}_1} \cup \overline{F^{\e}_3})$ and we set $\d := |F_1|-|F^{\e}_1|$. In to restore the prescribed values of the masses, we consider a point $x \in \R^n$ with $\theta_{F_2}(x) = 1$ (which exists, since $|F_2| > 0$). By Theorem \ref{thm:no-infiltration}, there exists $r > 0$ such that $|B_r(x) \cap F_2| = |B_r(x)|$. We take $\e > 0$ small enough, so that $\d < |B_r(x)|$ and we define $\tilde{F}^{\e}_1 := F^{\e}_1 \cup B_{r'}(x)$ (with $|B_{r'}(x)| = \d$), $\tilde{F}^{\e}_2 := F^{\e}_2 \setminus B_{r'}(x)$ and $\tilde{F}^{\e}_3 := F^{\e}_3$. Since we have assumed $\H^{n-1}(\partial^*F_1 \cap \partial^*F_3 \cap \Omega) > 0$, taking into account \eqref{eq:perim-approx} and \eqref{eq:violation_triangle}, we conclude that $\F_{\Omega,\mmm}(\tilde{F}^{\e}_1,\tilde{F}^{\e}_2,\tilde{F}^{\e}_3) < \F_{\Omega,\mmm}(F_1,F_2,F_3)$, 
which contradicts the minimality of
$(F_1,F_2,F_3)$.
\end{proof}
\begin{teo}\label{teoreg}
Let $\Omega \subset \R^n$ be an open set and let $\mmm:= (m_1,m_2)$, with $0<m_1,m_2<\infty$ and $m_1 + m_2 < |\Omega|$. If $(F_1,F_2,F_3) \in \MM(\F_{\Omega,\mmm})$, then for any $i\in\{1,2,3\}$ the set 
$\partial F_i$ is of class $C^\infty$ out of a closed singular set 
with zero $\H^{n-1}$ measure.
\end{teo}

\begin{proof}
By \eqref{eq:minimal-regularity} and Proposition \ref{prop:separation-phases} we have that $\H^{n-1}-$almost every  $x \in {\partial} F_1 \cap \Omega$ is an element of $\partial^* F_1\cap \partial^* F_2 \cap \Omega$. Thus, using Lemma \ref{lem:separation123}, there exist $r >0$ such that $B_r(x)\cap F_3 = \emptyset$ (with $B_r(x) \subset \subset \Omega$) and, by standard regularity theory for minimizing boundaries with prescribed volume (see \cite{Magg12}), we can conclude that $B_r(x)\cap \partial F_1 = B_r(x)\cap \partial F_2$ 
is $C^{\infty}$ on the complement of a closed set of Hausdorff dimension smaller or equal to $n-8$. In particular, ${\partial} F_1$ is $C^{\infty}$ on the complement of a closed set with zero $\H^{n-1}$ measure.\\
The same argument holds for the set $F_3$, and consequently ${\partial} F_3$ is $C^{\infty}$ on the complement of a closed set with zero $\H^{n-1}$ measure.\\
In particular, the sets ${\partial} F_2 \cap {\partial} F_1 \cap \Omega$ and $\partial F_2\cap {\partial} F_3 \cap \Omega$ are $C^{\infty}$ on the complement of a closed set with zero $\H^{n-1}$ measure. This implies that also $\partial F_2 \cap \Omega$ is $C^{\infty}$ on the complement of a closed set with zero $\H^{n-1}$ measure.
\end{proof}
\begin{Remark}
If we consider $\Omega := \R^2$ and $0<m_1,m_2<\infty$, it is easy to see that $\MM(\F_{\Omega,\mmm})$ is the set of the triples $(F_1,F_2,F_3)$, where $F_1$ is a metric ball with $|F_1| = m_1$, $F_2:= B^{(m_1 + m_2)}\setminus F_1$, where $B^{(m_1 + m_2)}$ is a metric ball of mass $m_1 + m_2$ which contains $F_1$, and $F^3 := \R^2 \setminus B^{(m_1 + m_2)}$. One of the possible minimizing configurations is realized when $F_1$ is tangent at a point (from the inside) to $F_2$. The point of contact between the two sets is not a point of $\partial^*F_2$ and this shows that even one dimensional minimizing boundaries for \eqref{eq:main_functional} are not necessarily everywhere regular.
\end{Remark}

\section{Nonexistence of minimizers in general domains}
In this final section we show that on domains $\Omega$ which do not satisfy Definition \ref{dfnz:isop-Foliation}, there are choices of $\mmm$ for which the infimum of $\F_{\Omega,\mmm}$ is not attained.
\begin{prop}
Let $\Omega = (0,1) \times (0,1) \subset \R^2$, $\s_{12}=\s_{23}=1$, $\s_{13} > 2$, $m_1 = \frac{\pi (1 + \e^2)}{16}$ (for $\e > 0$ small enough), $m_2 = 1/2 - \frac{\pi (1 + \e^2)}{16}$ and $m_3= 1/2$. The functional $\F_{\Omega , \mmm}$ has no minimum on $\C_{\Omega,\mmm}$.
\end{prop}

\begin{proof}
We argue by contradiction. If the infimum of $\F_{\Omega,\mmm}$ would be attained in correspondence of a triple $(F_1,F_2,F_3)$, by Proposition \ref{prop:separation-phases}, we would have that $\H^{n-1}(\partial^* F_1 \cap \partial^*F_3 \cap \Omega) = 0$ and, by standard regularity, the interfaces $\partial^* F_1 \cap \partial^* F_2 \cap \Omega $ and $\partial^* F_2 \cap \partial^* F_3 \cap \Omega$ would be either straight segments or circular arcs meeting $\partial \Omega$ orthogonally. If both $\partial^* F_1 \cap \partial^* F_2 \cap \Omega$ and $\partial^* F_2 \cap \partial^* F_3 \cap \Omega$ were segments, the value of the minimum of \eqref{eq:main_functional} would be $2$. This is a contradiction to minimality, since, for a sufficiently small $\e > 0$, $2 > 1 + \frac{\sqrt{\pi(1 + \e)}}{4}$ and $ 1 + \frac{\sqrt{\pi(1 + \e)}}{4}$ is the value attained by $\F_{\Omega , \mmm}$ for $F_1 = (\{x^2 + y^2 < 1/2\} \cup \{x^2 + (y-1)^2 = \e/2 \})\cap \Omega$, $F_3 = \Omega \cap  \{(x,y) \in \R^2, x > 1/2\}$ and $F_2 = \Omega \setminus \overline{(F_1 \cup F_3)}$. If $F_1$ and $F_3$ would be quarter of disks centred at different corners of $\Omega$, it is easy to see that they would overlap and consequently they could not be element of a triple in $\C_{\Omega,\mmm}$. If we set $F_1 = \{x^2 + y^2 < \frac{1+\e^2}{4}\}$, $F_3 = \Omega \setminus \{x^2 + y^2 < \sqrt{\frac{2}{\pi}}\}$ and $F_2 = \Omega \setminus \overline{(F_1 \cup F_3)}$, we still have $\F_{\Omega , \mmm}(F_1 , F_2 , F_3) = \sqrt{\frac{\pi}{2}} + \frac{\pi}{4} > 2$ and $F_2 = $. As a consequence, $\F_{\Omega , \mmm}$ has no minimum on $\C_{\Omega,\mmm}$.
\end{proof}


\begin{thebibliography}{10}
%
\bibitem{ATWa93}
F.J. Almgren, J. Taylor, L. Wang;
\newblock Curvature--driven flows: a variational approach.
\newblock {\em SIAM J. Control Optim.} 31(1993), 387--437.
%
\bibitem{AmBr90}
L. Ambrosio, A. Braides;
\newblock Functionals defined on partitions in sets of finite perimeter. II.
\newblock {\em J. Math. Pures Appl.} 69(1990), 307--333.
%
\bibitem{AFPa00}
L. Ambrosio, N. Fusco, D. Pallara;
\newblock Functions of bounded variation and free discontinuity problems.
\newblock {\em Oxford Mathematical Monographs}, The Clarendon Press, Oxford University Press, New York, 2000.
%
\bibitem{ANPa02}
L. Ambrosio, M. Novaga, E. Paolini;
\newblock Some Regularity Results for Minimal Crystals.
\newblock {\em ESAIM: COCV} 8(2002), 69--103.
%
\bibitem{BeMa15}
E. Bretin, S. Masnou;
\newblock A new phase field model for inhomogeneous minimal partitions, and applications to droplets dynamics.
\newblock {\em Preprint}, 2015.
%
\bibitem{EsOt14}
S. Esedoglu, F. Otto;
\newblock Threshold dynamics for networks with arbitrary surface tensions.
\newblock {\em Comm. Pure Appl. Math.}, 68(2015), 808–864.
%
\bibitem{Leon01}
G.P. Leonardi;
\newblock Infiltrations in immiscible fluids systems.
\newblock {\em Proc. Roy Soc. Edimburgh}, 131A(2001), 425-436.
%
\bibitem{Magg12}
F. Maggi;
\newblock Sets of Finite Perimeter and Geometric Variational Problems.
\newblock {\em Cambridge studies in advanced mathematics}, 135(2012).
%
\bibitem{Morg97}
F. Morgan;
\newblock Lower semicontinuity of energy clusters.
\newblock {\em Proc. Roy.  Soc. Edinburgh}, 127A(1997), 819--822.
%
\bibitem{Morg98}
F. Morgan;
\newblock Immiscible fluid clusters in {${\bf R}^2$} and {${\bf R}^3$}.
\newblock {\em Michigan Math. J.}, 45(1998), 441--450.
%
\bibitem{RiVe15}
M. Ritor\'e; E. Vernadakis;
\newblock Isoperimetric inequalities in Euclidean convex bodies.
\newblock {\em Trans. Amer. Math. Soc.}, 367(2015), 4983--5014.
%
\bibitem{Schm15}
T. Schmidt;
\newblock Strict interior approximation of sets of finite perimeter and functions of bounded variation.
\newblock {\em Proc. Am. Math. Soc.}, 143(2015), 2069--2084.
%
\bibitem{VoDe10}
A. Vo{\ss}-B\"ohme, A. Deutsch;
\newblock The cellular basis of cell sorting kinetics.
\newblock {\em J. Theoret. Biol.}, 263(2010), 419--436.
%
\bibitem{Whit96}
B. White;
\newblock Existence of least-energy configurations of immiscible fluids.
\newblock {\em J. Geom. Analysis}, 6(1996), 151--161.
\end{thebibliography}
\end{document}